\documentclass[fleqn,12pt,twoside]{article}

\usepackage[headings]{espcrc1}
\readRCS
$Id: espcrc1.tex,v 1.2 2004/02/24 11:22:11 spepping Exp $
\usepackage{graphicx}
\usepackage[figuresright]{rotating}

\usepackage{amsmath}
\usepackage{amsfonts}
\usepackage{amssymb}
\usepackage{amsthm}

\usepackage{tikz}
\usetikzlibrary{arrows}
\usetikzlibrary{shapes}
\usetikzlibrary{decorations.markings}

\newtheorem{theorem}{Theorem}

\newtheorem{conjecture}[theorem]{Conjecture}
\newtheorem{corollary}[theorem]{Corollary}

\newtheorem{lemma}[theorem]{Lemma}

\newtheorem{problem}{Problem}

\newenvironment{acknowledgement}[1][Acknowledgement]{\begin{trivlist}
\item[\hskip \labelsep {\bfseries #1}]}{\end{trivlist}}

\numberwithin{theorem}{section}

\title{Further results on the deficiency of graphs}

\author{P.A. Petrosyan\address[MCSD]{Department of Informatics and Applied Mathematics,\\
Yerevan State University, 0025, Armenia}%
\address{Institute for Informatics and Automation Problems,\\
National Academy of Sciences, 0014, Armenia}%
\thanks {email: pet\_petros@ipia.sci.am, petros\_petrosyan@ysu.am},
        H.H. Khachatrian\addressmark[MCSD]%
\thanks{email: hrant.khachatrian@ysu.am}}

\runtitle{Further results on the deficiency of graphs}\runauthor{P.A. Petrosyan, H.H. Khachatrian}

\begin{document}

% typeset front matter
\maketitle

\begin{abstract}
A \emph{proper $t$-edge-coloring} of a graph $G$ is a mapping
$\alpha: E(G)\rightarrow \{1,\ldots,t\}$ such that all colors are
used, and $\alpha(e)\neq \alpha(e^{\prime})$ for every pair of
adjacent edges $e,e^{\prime}\in E(G)$. If $\alpha $ is a proper edge-coloring of a graph $G$ and $v\in
V(G)$, then \emph{the spectrum of a vertex $v$}, denoted by $S\left(v,\alpha \right)$, is the set of all colors appearing on edges incident to $v$. \emph{The deficiency of $\alpha$ at vertex $v\in V(G)$}, denoted by $def(v,\alpha)$, is the minimum number of integers which must be added to $S\left(v,\alpha \right)$ to form an interval, and \emph{the deficiency $def\left(G,\alpha\right)$ of a proper edge-coloring $\alpha$ of $G$} is defined as the sum $\sum_{v\in V(G)}def(v,\alpha)$. \emph{The deficiency of a graph $G$}, denoted by $def(G)$, is defined as follows:
$def(G)=\min_{\alpha}def\left(G,\alpha\right)$, where minimum is
taken over all possible proper edge-colorings of $G$. For a
graph $G$, the smallest and the largest values of $t$ for which it
has a proper $t$-edge-coloring $\alpha$ with deficiency
$def(G,\alpha)=def(G)$ are denoted by $w_{def}(G)$ and $W_{def}(G)$, respectively. In this paper, we obtain some bounds on $w_{def}(G)$ and $W_{def}(G)$. In particular, we show that for any $l\in \mathbb{N}$, there exists a graph $G$ such that $def(G)>0$ and $W_{def}(G)-w_{def}(G)\geq l$. It is known that for the complete graph $K_{2n+1}$, $def(K_{2n+1})=n$ ($n\in \mathbb{N}$). Recently, Borowiecka-Olszewska, Drgas-Burchardt and Ha\l uszczak posed the following conjecture on the deficiency of near-complete graphs: if $n\in \mathbb{N}$, then $def(K_{2n+1}-e)=n-1$. In this paper, we confirm this conjecture.\\

Keywords: proper edge-coloring, interval (consecutive) coloring,
deficiency, complete graph, near-complete graph.
\end{abstract}

\begin{center}
Dedicated to the memory of Haroutiun Khachatrian
\end{center}

\section{Introduction}\

All graphs considered in this paper are finite, undirected, and have no loops or multiple edges. Let $V(G)$ and $E(G)$ denote the sets of vertices and edges of $G$, respectively. The degree of a vertex $v\in V(G)$ is denoted by $d_{G}(v)$, the diameter of $G$ by $\mathrm{diam}(G)$, and the chromatic index of $G$ by
$\chi^{\prime}(G)$. For a graph $G$, let $\Delta(G)$ and $\delta(G)$ denote the maximum and minimum degrees of vertices in $G$, respectively. The terms and concepts that we do not define can be found in \cite{AsrDenHag,Kubale,West}.

A \emph{proper $t$-edge-coloring} of a graph $G$ is a mapping
$\alpha: E(G)\rightarrow \{1,\ldots,t\}$ such that all colors are
used, and $\alpha(e)\neq \alpha(e^{\prime})$ for every pair of
adjacent edges $e,e^{\prime}\in E(G)$. If $\alpha $ is a proper edge-coloring of a graph $G$ and $v\in
V(G)$, then \emph{the spectrum of a vertex $v$}, denoted by $S\left(v,\alpha \right)$, is the set of all colors appearing on edges incident to $v$. A proper $t$-edge-coloring $\alpha$ of a graph $G$ is an \emph{interval $t$-coloring} if for each vertex $v$ of $G$, the set $S\left(v,\alpha \right)$ is an interval of
integers. A graph $G$ is \emph{interval colorable} if it has an
interval $t$-coloring for some positive integer $t$. The set of all interval colorable graphs is denoted by $\mathfrak{N}$. For a graph $G\in \mathfrak{N}$, the smallest and the largest values of $t$ for which it has an interval $t$-coloring are denoted by $w(G)$ and $W(G)$, respectively. The concept of interval edge-coloring of graphs was introduced by Asratian and Kamalian \cite{AsrKam} in 1987. In \cite{AsrKam}, the authors proved that if $G\in \mathfrak{N}$, then $\chi^{\prime}\left(G\right)=\Delta(G)$.
Asratian and Kamalian also proved \cite{AsrKam,AsrKamJCTB} that if a triangle-free graph $G$ admits an interval $t$-coloring, then $t\leq \left\vert V(G)\right\vert -1$. In \cite{Kampreprint,KamDiss},
Kamalian investigated interval colorings of complete bipartite
graphs and trees. In particular, he proved that the complete
bipartite graph $K_{m,n}$ has an interval $t$-coloring if and only
if $m+n-\gcd(m,n)\leq t\leq m+n-1$, where $\gcd(m,n)$ is the
greatest common divisor of $m$ and $n$. In \cite{PetDM,PetKhachTan},
Petrosyan, Khachatrian and Tananyan proved that the $n$-dimesional
cube $Q_{n}$ has an interval $t$-coloring if and only if $n\leq t\leq \frac{n\left(n+1\right)}{2}$. The problem of determining whether or not a given graph is interval colorable is $NP$-complete, even for regular \cite{AsrKam} and bipartite \cite{Seva} graphs. In some papers
\cite{AsrKam,AsrKamJCTB,GiaroKubale1,GiaroKubale2,Hansen,HansonLotenToft,Kampreprint,KamDiss,Kubale,PetDM,PetKarapet,PetKhachTan,PetKhachYepTan,PetKhachMam,Seva},
the problems of the existence, construction and estimating of the numerical parameters of interval colorings of graphs were investigated.

It is known that there are graphs that have no interval colorings. A smallest example is $K_{3}$. Since not all graphs admit an
interval coloring, it is naturally to consider a measure of
closeness for a graph to be interval colorable. In
\cite{GiaroKubaleMalaf1}, Giaro, Kubale and Ma\l afiejski introduced such a measure which is called deficiency of a graph (another measure was suggested in \cite{PetSarg}). The
\emph{deficiency $def(G)$ of a graph $G$} is the minimum number of
pendant edges whose attachment to $G$ makes it interval colorable.
The concept of deficiency of graphs can be also defined using proper edge-colorings. \emph{The deficiency of a proper edge-coloring $\alpha$ at vertex $v\in V(G)$}, denoted by $def(v,\alpha)$, is the minimum number of integers which must be added to $S\left(v,\alpha \right)$ to form an interval, and \emph{the deficiency $def\left(G,\alpha\right)$ of a proper edge-coloring $\alpha$ of $G$} is defined as the sum $\sum_{v\in V(G)}def(v,\alpha)$. In fact, $def(G)=\min_{\alpha}def\left(G,\alpha\right)$, where minimum is
taken over all possible proper edge-colorings of $G$.
The problem of determining the deficiency of a graph is
$NP$-complete, even for regular and bipartite graphs
\cite{AsrKam,Seva,GiaroKubaleMalaf1}. In \cite{GiaroKubaleMalaf1},
Giaro, Kubale and Ma\l afiejski obtained some results on the
deficiency of bipartite graphs. In particular, they showed that
there are bipartite graphs whose deficiency approaches the number of vertices.  In \cite{GiaroKubaleMalaf2}, the same authors proved that if $G$ is an $r$-regular graph with an odd number of vertices, then $def(G)\geq \frac{r}{2}$, and determined the deficiency of odd cycles, complete graphs, wheels and broken wheels. In \cite{Schwartz}, Schwartz investigated the deficiency of regular graphs. In particular, he obtained tight bounds on the deficiency of regular graphs and proved that there are regular graphs with high deficiency. Bouchard, Hertz and Desaulniers \cite{BouchHertzDesau}
derived some lower bounds for the deficiency of graphs and provided a tabu search algorithm for finding a proper edge-coloring with minimum deficiency of a
graph. Recently, Borowiecka-Olszewska,
Drgas-Burchardt and Ha\l uszczak \cite{B-OD-BHal} studied the
deficiency of $k$-trees. In particular, they determined the
deficiency of all $k$-trees with maximum degree at most $2k$, where
$k\in \{2,3,4\}$. They also proved that the following lower bound
for $def(G)$ holds: if $G$ is a graph with an odd number of
vertices, then $def(G)\geq \frac{2\vert E(G)\vert -(\vert V(G)\vert
-1)\Delta(G)}{2}$. In the same paper, Borowiecka-Olszewska,
Drgas-Burchardt and Ha\l uszczak posed the following conjecture on
the deficiency of near-complete graphs: if $n\in \mathbb{N}$, then
$def(K_{2n+1}-e)=n-1$.

For a graph $G$, the smallest and the largest values of $t$ for which it has a proper $t$-edge-coloring $\alpha$ with deficiency
$def(G,\alpha)=def(G)$ are denoted by $w_{def}(G)$ and $W_{def}(G)$, respectively. In this paper, we obtain some bounds on $w_{def}(G)$ and $W_{def}(G)$. We also determine the deficiency of certain graphs. In particular, we confirm the above-mentioned conjecture of Borowiecka-Olszewska, Drgas-Burchardt and Ha\l uszczak.\\

\section{Notation, definitions and auxiliary results}\

We use standard notation $C_{n}$ and $K_{n}$ for the simple cycle
and complete graph on $n$ vertices, respectively. We also use
standard notation $K_{m,n}$ and $K_{l,m,n}$ for the complete
bipartite and tripartite graph, respectively, one part of which has
$m$ vertices, other part has $n$ vertices and a third part has $l$
vertices.

For two positive integers $a$ and $b$ with $a\leq b$, we denote by
$\left[a,b\right]$ the interval of integers
$\left\{a,a+1,\ldots,b-1,b\right\}$. If $a>b$, then $\left[a,b\right]=\emptyset$.

Let $A$ be a finite set of integers. The deficiency $def(A)$ of $A$
is the number of integers between $\min A$ and $\max A$ not
belonging to $A$. Clearly, $def(A)=\max A-\min A-\vert A\vert +1$. A set $A$ with $def(A)=0$ is an interval.

If $\alpha $ is a proper edge-coloring of a graph $G$ and $v\in
V(G)$, then the spectrum of a vertex $v$, denoted by $S\left(v,\alpha \right)$, is the set of colors appearing on edges incident to $v$. The smallest and largest colors
of $S\left(v,\alpha \right)$ are denoted by $\underline
S\left(v,\alpha \right)$ and $\overline S\left(v,\alpha \right)$,
respectively. If $\alpha $ is a proper edge-coloring of a graph $G$ and $V^{\prime}\subseteq V(G)$, then we can define $S\left(V^{\prime},\alpha\right)$ as follows: $S\left(V^{\prime},\alpha\right)=\bigcup_{v\in V^{\prime}}S\left(v,\alpha \right)$. The smallest and largest colors of $S\left(V^{\prime},\alpha\right)$ are denoted by $\underline S\left(V^{\prime},\alpha\right)$ and $\overline S\left(V^{\prime},\alpha\right)$, respectively. The deficiency of $\alpha$ at vertex $v\in V(G)$,
denoted by $def(v,\alpha)$, is defined as follows:
$def(v,\alpha)=def\left(S\left(v,\alpha \right)\right)$. The
deficiency of a proper edge-coloring $\alpha$ of $G$ is defined as
the sum $def\left(G,\alpha\right)=\sum_{v\in V(G)}def(v,\alpha)$.
For a graph $G$, define the deficiency $def(G)$ as follows:
$def(G)=\min_{\alpha}def\left(G,\alpha\right)$, where minimum is
taken over all possible proper edge-colorings of $G$. Clearly,
$def(G)=0$ if and only if $G\in\mathfrak{N}$. Also, it is easy to
see that $def(G)$ can be defined as the minimum number of pendant
edges whose attachment to $G$ makes it interval colorable. For a
graph $G$, the smallest and the largest values of $t$ for which it
has a proper $t$-edge-coloring $\alpha$ with deficiency
$def(G,\alpha)=def(G)$ are denoted by $w_{def}(G)$ and $W_{def}(G)$,
respectively. Clearly, if for a graph $G$, $def(G)=0$, then
$w_{def}(G)=w(G)$ and $W_{def}(G)=W(G)$.\\

We will use the following five results.

\begin{theorem}
\label{mytheorem1.1}\cite{AsrKam,AsrKamJCTB} If $G\in \mathfrak{N}$, then $\chi^{\prime}(G)=\Delta(G)$. Moreover, if $G$ is a regular graph, then $G\in \mathfrak{N}$ if and only if
$\chi^{\prime}(G)=\Delta(G)$.
\end{theorem}

\begin{theorem}
\label{mytheorem1.2}\cite{AsrKam,AsrKamJCTB} If $G$ is a regular graph and $G\in \mathfrak{N}$, then for every $t$, $w(G) \leq t \leq W(G)$, $G$ has an interval $t$-coloring.
\end{theorem}

\begin{theorem}
\label{mytheorem1.3}\cite{AsrKam,AsrKamJCTB} If $G$ is a
triangle-free graph and $G\in \mathfrak{N}$, then $W(G)\leq \vert
V(G)\vert -1$.
\end{theorem}

\begin{theorem}
\label{mytheorem1.4}\cite{Axen} If $G$ is a planar graph and $G\in \mathfrak{N}$, then $W(G)\leq \frac{11}{6}\vert V(G)\vert$.
\end{theorem}

\begin{theorem}
\label{mytheorem1.5}\cite{GrzesikKhach} For any $m,n\in \mathbb{N}$, $K_{1,m,n}\in \mathfrak{N}$ if and only if $\gcd(m+1,n+1)=1$. Moreover, if $\gcd(m+1,n+1)=1$, then $w\left(K_{1,m,n}\right)=m+n$.
\end{theorem}

We also need the following lemma on a special interval coloring of the complete bipartite graph $K_{p,p}$.

\begin{lemma}
\label{Kpp}
If $K_{p,p}$ is a complete bipartite graph with bipartition $(X,Y)$, where $X=\{x_{1},\ldots,x_{p}\}$ and $Y=\{y_{1},\ldots,y_{p}\}$, then $K_{p,p}$ has an interval coloring $\beta$ such that $\underline{S}(x_i,\beta) = \underline{S}(y_i,\beta) = \left\lfloor \frac{i}{2}\right\rfloor + 1$ for $1\leq i\leq p$.
\end{lemma}
\begin{proof}
We construct an edge-coloring $\beta$ of $K_{p,p}$ in the following way: for every $1\leq k\leq\left\lfloor\frac{p}{2} \right\rfloor$, let

\begin{center}
$\beta\left(x_{i}y_{j}\right)=\left\{
\begin{tabular}{ll}
$k$, & if $i+j=2k$,\\
$k+p$, & if $i+j=2k+p$.\\
\end{tabular}%
\right.$
\end{center}

At this point exactly $\left\lfloor \frac{p}{2} \right\rfloor$ edges incident to each vertex are colored. Let $H$ denote the subgraph of $K_{p,p}$ which contains all vertices of $K_{p,p}$ but only its non-colored edges. Clearly, $H$ is a $\left\lceil \frac{p}{2}\right\rceil$-regular bipartite graph, so, by K\"onig's edge-coloring theorem, it has a proper $\left\lceil \frac{p}{2}\right\rceil$-edge-coloring $\beta'$. We complete the edge-coloring $\beta$ of $K_{p,p}$ by shifting the colors from $\beta'$ by $\left\lfloor \frac{p}{2} \right\rfloor$. For every $e \in E(H)$, let
\begin{align*}
\beta(e) = \beta'(e) + \left\lfloor \frac{p}{2} \right\rfloor.
\end{align*}

The resulting spectrums of the vertices $x_i$ and $y_i$ ($1\leq i\leq p$) are the following:
\begin{align*}
    S(x_i,\beta) = S(y_i,\beta) = \left[\left\lfloor \frac{i}{2} \right\rfloor + 1,\left\lfloor \frac{p}{2} \right\rfloor \right] \cup \left[ 1+p, \left\lfloor \frac{i}{2} \right\rfloor + p \right] \cup \left[ \left\lfloor \frac{p}{2} \right\rfloor + 1, p \right],
\end{align*}
where the last interval comes from the subgraph $H$ and is common for all vertices of $K_{p,p}$. The union of these intervals is
\begin{align*}
    S(x_i,\beta) = S(y_i,\beta) = \left[\left\lfloor \frac{i}{2} \right\rfloor + 1, \left\lfloor \frac{i}{2} \right\rfloor + p \right] (1\leq i\leq p).
\end{align*}
\end{proof}

Note that this lemma is a partial case of Lemma 4 from \cite{TePet}.
\bigskip

\section{Some bounds on $w_{def}(G)$ and $W_{def}(G)$}\

Recently, in \cite{AltinCaporHertz}, Altinakar, Caporossi and Hertz proved that if $G$ has at least three vertices, then $W_{def}(G)\leq 2\vert V(G)\vert -4+def(G)$. Here, we provide some other bounds for the parameters $w_{def}(G)$ and $W_{def}(G)$ depending on the number of vertices, degrees and diameter for connected, triangle-free and planar graphs.

\begin{theorem}
\label{mytheorem2.1} Let $\mathfrak{C}$ be a class of graphs closed under the operation of attaching of pendant edges. If
$W(G^{\prime})\leq
f(\vert V(G^{\prime})\vert)$ holds for any graph $G^{\prime}\in
\mathfrak{C} \cap \mathfrak{N}$, then for any graph $G\in \mathfrak{C}$, we have
\begin{center}
$W_{def}(G)\leq f\left(\vert V(G)\vert+def(G)\right)$.
\end{center}
\end{theorem}
\begin{proof}
Let $G\in \mathfrak{C}$ and $\alpha$ be a proper $W_{def}(G)$-edge-coloring of $G$ such that $def(G,\alpha)=def(G)$.

Define an auxiliary graph $G^{\prime}$ as follows: for each vertex $v\in V(G)$ with $def(v,\alpha)>0$, we attach $def(v,\alpha)$ pendant edges at vertex $v$. Clearly, $G^{\prime}\in\mathfrak{C}$ and $\vert V(G^{\prime})\vert = \vert V(G)\vert +def(G)$. Next, we extend a proper $W_{def}(G)$-edge-coloring $\alpha$ of $G$ to a proper $W_{def}(G)$-edge-coloring $\beta$ of $G^{\prime}$ as follows: for each vertex $v\in V(G)$ with $def(v,\alpha)>0$, we color the attached edges incident to $v$ using distinct colors from $\left[\underline
S\left(v,\alpha \right),\overline S\left(v,\alpha
\right)\right]\setminus S(v,\alpha)$. By the definition of $\beta$
and the construction of $G^{\prime}$, we obtain that $G^{\prime}$ has an interval $W_{def}(G)$-coloring. Since $G^{\prime}\in \mathfrak{C}\cap\mathfrak{N}$, we have

\begin{center}
$W_{def}(G)\leq W(G^{\prime})\leq f(\vert V(G^{\prime})\vert)= f\left(\vert V(G)\vert+def(G)\right)$.
\end{center}
\end{proof}

Since the class of triangle-free graphs is closed under the operation of attaching of pendant edges, by Theorems \ref{mytheorem1.3} and \ref{mytheorem2.1}, we obtain the following results.

\begin{corollary}
\label{mycorollary2.1} If $G$ is a triangle-free graph, then
$W_{def}(G)\leq \vert V(G)\vert +def(G)-1$.
\end{corollary}

\begin{corollary}
\label{mycorollary2.2} If $G$ is a bipartite graph, then
$W_{def}(G)\leq \vert V(G)\vert +def(G)-1$.
\end{corollary}

It is known that $W_{def}(K_{m,n})=m+n-1$ ($m,n\in \mathbb{N}$)
\cite{Kampreprint,KamDiss}, so this upper bound is sharp for graphs $G$ with $def(G)=0$. Now let us consider graphs $G$ with $def(G)=1$.
Let $K_{n,n}^{\prime}$ be a complete bipartite graph with exactly
one subdivided edge, i.e.
\begin{center}
$V\left(K_{n,n}^{\prime}\right)=\{u_{1},\ldots,u_{n},v_{1},\ldots,v_{n},w\}$
and\\
$E\left(K_{n,n}^{\prime}\right)=\left(\{u_{i}v_{j}\colon\,1\leq
i\leq n,1\leq j\leq n\}\setminus \{u_{n}v_{n}\}\right)\cup
\{u_{n}w,wv_{n}\}$.
\end{center}

It is well-known that
$\chi^{\prime}\left(K_{n,n}^{\prime}\right)=n+1$
\cite{BeinekeWilson,Vizing}, so by Theorem \ref{mytheorem1.1},
$K_{n,n}^{\prime}\notin \mathfrak{N}$. Thus,
$def\left(K_{n,n}^{\prime}\right)\geq 1$. On the other hand, let us define an edge-coloring $\alpha$ of $K_{n,n}^{\prime}$ as follows: for every $e\in E\left(K_{n,n}^{\prime}\right)$, let

\begin{center}
$\alpha\left(e\right)=\left\{
\begin{tabular}{ll}
$i+j-1$, & if $e=u_{i}v_{j}$,\\
$2n-1$, & if $e=u_{n}w$,\\
$2n$, & if $e=wv_{n}.$\\
\end{tabular}%
\right.$
\end{center}

It is not difficult to see that $\alpha$ is a proper
$2n$-edge-coloring of $K_{n,n}^{\prime}$ with deficiency
$def(K_{n,n}^{\prime},\alpha)= 1$. Hence, $def(K_{n,n}^{\prime})=1$
and $2n\leq W_{def}(K_{n,n}^{\prime})\leq 2n+1$.\\

Since the class of planar graphs is closed under the operation of attaching of pendant edges, by Theorems \ref{mytheorem1.4} and \ref{mytheorem2.1}, we obtain the following result.

\begin{corollary}
\label{mycorollary2.3} If $G$ is a planar graph, then
$W_{def}(G)\leq \frac{11}{6}\left(\vert V(G)\vert +def(G)\right)$.
\end{corollary}

Next we give some upper bounds for $W_{def}(G)$ depending on degrees and diameter of the connected graph $G$.

\begin{theorem}
\label{mytheorem2.2} If $G$ is a connected graph, then
\begin{center}
$W_{def}(G)\leq 1+def(G)+{\max\limits_{P\in
\mathbf{P}}}{\sum\limits_{v\in V(P)}}\left(d_{G}(v)-1\right)$,
\end{center}
where $\mathbf{P}$ is the set of all shortest paths in the graph
$G$.
\end{theorem}
\begin{proof} In the proof we follow the idea from \cite{AsrKamJCTB} (Thm. 2).
Let $\alpha$ be a proper $W_{def}(G)$-edge-coloring of $G$ such that $def(G,\alpha)=def(G)$.

Similarly as in the proof of Theorem \ref{mytheorem2.1}, we define an auxiliary graph $H$ as follows: for each vertex $v\in V(G)$ with $def(v,\alpha)>0$, we attach $def(v,\alpha)$ pendant edges at vertex $v$. Clearly, $H$ is a connected graph. Again, we extend a proper $W_{def}(G)$-edge-coloring $\alpha$ of $G$ to a proper
$W_{def}(G)$-edge-coloring $\beta$ of $H$ as follows: for each
vertex $v\in V(G)$ with $def(v,\alpha)>0$, we color the attached
edges incident to $v$ using colors from $\left[\underline
S\left(v,\alpha \right),\overline S\left(v,\alpha
\right)\right]\setminus S(v,\alpha)$. By the definition of $\beta$
and the construction of $H$, we obtain that $H$ has an interval
$W_{def}(G)$-coloring. In the coloring $\beta$ of $H$, we consider
the edges with colors $1$ and $W_{def}(G)$. Let $e=u_{1}u_{2},
e^{\prime}=w_{1}w_{2}$ and $\beta(e)=1,
\beta(e^{\prime})=W_{def}(G)$. Without loss of generality we may
assume that a shortest path $P$ joining $e$ with $e^{\prime}$ joins
$u_{1}$ with $w_{1}$, where $P=v_{0}e_{1}v_{1}\ldots
v_{i-1}e_{i}v_{i}\ldots v_{k-1}e_{k}v_{k}$ and $v_{0}=u_{1}$,
$v_{k}=w_{1}$.

Since $\beta$ is an interval $W_{def}(G)$-coloring of $H$, we have

\begin{center}
$\beta(e_{1})\leq d_{H}(v_{0})$,

$\beta(e_{2})\leq \beta(e_{1})+d_{H}(v_{1})-1$,

$\cdots \cdots \cdots \cdots \cdots \cdots$

$\beta(e_{i})\leq \beta(e_{i-1})+d_{H}(v_{i-1})-1$,

$\cdots \cdots \cdots \cdots \cdots \cdots$

$\beta(e_{k})\leq \beta(e_{k-1})+d_{H}(v_{k-1})-1$,

$\beta(e^{\prime})\leq \beta(e_{k})+d_{H}(v_{k})-1$.

\end{center}

Summing up these inequalities, we obtain

\begin{center}
$\beta(e^{\prime})\leq
1+{\sum\limits_{j=0}^{k}\left(d_{H}(v_{j})-1\right)}$.
\end{center}

From here, we obtain

\begin{eqnarray*}
W_{def}(G)=\beta(e^{\prime}) &\leq&
1+{\sum\limits_{j=0}^{k}\left(d_{H}(v_{j})-1\right)}\leq
1+def(G)+{\sum\limits_{j=0}^{k}\left(d_{G}(v_{j})-1\right)}\\
&\leq& 1+def(G)+{\max\limits_{P\in \mathbf{P}}}{\sum\limits_{v\in
V(P)}}\left(d_{G}(v)-1\right).
\end{eqnarray*}
\end{proof}

\begin{corollary}
\label{mycorollary2.4} If $G$ is a connected graph, then
\begin{center}
$W_{def}(G)\leq
1+def(G)+(\mathrm{diam}(G)+1)\left(\Delta(G)-1\right)$.
\end{center}
\end{corollary}

It is known that the upper bound in Theorem \ref{mytheorem2.2} is
sharp for trees \cite{Kampreprint,KamDiss}. The upper bound in
Corollary \ref{mycorollary2.4} cannot be significantly improved, since
$def\left(C_{2n+1}\right)=1$, $\mathrm{diam}\left(C_{2n+1}\right)=n$
and $W_{def}\left(C_{2n+1}\right)=n+2$ ($n\in \mathbb{N}$).\\

Here we prove a lower bound on $w_{def}(G)$ for graphs without
perfect matching.

\begin{theorem}
\label{mytheorem2.3} If $G$ has no perfect matching, then
$w_{def}(G)\geq 2\delta(G)-def(G)$.
\end{theorem}
\begin{proof} In the proof of this theorem we follow the idea from \cite{BouchHertzDesau} (Prop. 2).

Let $\alpha$ be a proper $w_{def}(G)$-edge-coloring of $G$ such that
$def(G,\alpha)=def(G)$.

It is easy to see that for every $v\in V(G)$, we have

\begin{center}
$1\leq \underline S\left(v,\alpha \right)\leq
w_{def}(G)-\delta(G)+1$.
\end{center}

Assume that $w_{def}(G)-\delta(G)+1\leq \delta(G)$ (otherwise,
$w_{def}(G)\geq 2\delta(G)\geq 2\delta(G)-def(G)$).

Since $G$ has no perfect matching, for each color $c\in
\left[w_{def}(G)-\delta(G)+1,\delta(G)\right]$, there is at least
one vertex $v_{c}$ such that $c\notin S\left(v_{c},\alpha\right)$
and $\underline S\left(v_{c},\alpha \right)<c<\overline
S\left(v_{c},\alpha \right)$. This implies that there are at least
$2\delta(G)-w_{def}(G)$ missing colors at vertices of $G$. Hence,
$def(G)\geq 2\delta(G)-w_{def}(G)$.
\end{proof}

\begin{corollary}
\label{mycorollary2.5}\cite{BouchHertzDesau} If $G$ is a graph with an odd number of vertices, then $w_{def}(G)\geq 2\delta(G)-def(G)$.
\end{corollary}

\begin{corollary}
\label{mycorollary2.6}\cite{BouchHertzDesau} If $G$ is an $r$-regular graph with an odd number of vertices and $def(G)=\frac{r}{2}$, then
$w_{def}(G)\geq \frac{3r}{2}$.
\end{corollary}

\begin{corollary}
\label{mycorollary2.7} If $G$ has no perfect matching and $G\in
\mathfrak{N}$, then $w(G)\geq \max\{\Delta(G),2\delta(G)\}$.
\end{corollary}

In \cite{GiaroKubaleMalaf2}, Giaro, Kubale and Ma\l afiejski
determined the deficiency of $K_{2n+1}$, and they proved that
$def(K_{2n+1})=n$ ($n\in \mathbb{N}$). The lower bound in Theorem
\ref{mytheorem2.3} is sharp for $K_{2n+1}$, since
$\delta(K_{2n+1})=2n$ and it was proved in \cite{BouchHertzDesau}
that $w_{def}\left(K_{2n+1}\right)=3n$ ($n\in \mathbb{N}$). In the
next section we will prove that this lower bound is also sharp for
near-complete graphs, but first we consider the difference $W_{def}(G)-w_{def}(G)$. In \cite{KamDiss}, Kamalian proved that for any $l\in \mathbb{N}$, there exists a graph $G$ such that $G\in \mathfrak{N}$ and $W(G)-w(G)\geq l$. Here we extend this result to graphs with a positive deficiency.

\begin{theorem}
\label{mytheorem2.4} For any $l\in \mathbb{N}$, there exists a graph $G$ such that $def(G)>0$ and $W_{def}(G)-w_{def}(G)\geq l$.
\end{theorem}
\begin{proof}
Let $p=2l+1$ and $G=K_{2p+1}$. Since $def(G)=p>0$ and $w_{def}(G)=3p$, so to complete the proof it is sufficient to show that $W_{def}(G) \geq 3p + \frac{p-1}{2} = 3p + l$. We will prove a more general statement for all complete graphs with odd number of vertices. We will show that if $n = p2^q$, where $p$ is odd and $q \in \mathbb{Z}_+$, then
\begin{center}
$W_{def}(K_{2n+1}) \geq 3n + \frac{p-1}{2}$.
\end{center}

\begin{figure}[t!]
\centering
\includegraphics[width=0.7\textwidth]{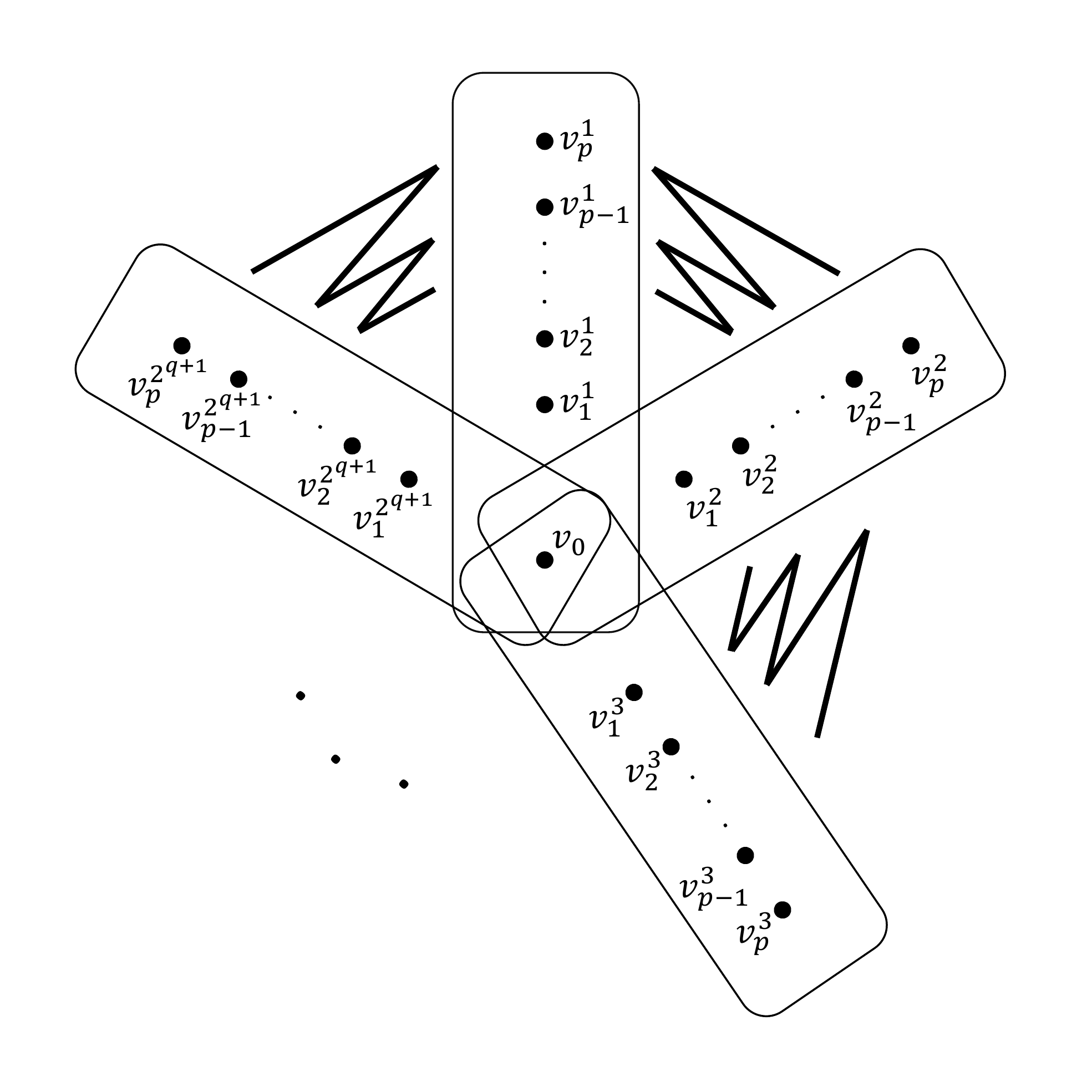}
\caption{We divide the vertices of $K_{p2^{q+1}+1}$ into $2^{q+1}$ groups, such that the intersection of each pair of groups is the vertex $v_0$.}
\label{complete-graph}
\end{figure}

We construct an edge-coloring $\phi$ of $K_{p2^{q+1}+1}$ by composing three colorings of different graphs.

\begin{enumerate}
    \item We need interval $\left(\frac{3}{2}(p+1)-2 \right)$-coloring $\alpha$ of $K_{p+1}$ as described in the proof of Theorem 4 from \cite{PetDM}. We denote the vertices of $K_{p+1}$ in the following way: $V(K_{p+1}) = \{u_i : i=0, \ldots, p\}$. The coloring $\alpha$ has an important property that $\underline{S}(u_i, \alpha) = \lfloor \frac{i}{2} \rfloor + 1$ for every $i = 0,\ldots,p$.
    \item Next, we need the interval coloring $\beta$ of $K_{p,p}$ described in Lemma \ref{Kpp}. We denote the vertices of $K_{p,p}$ in the following way: $V(K_{p,p}) = \{x_i,y_i:i=1,\ldots,p\}$. Lemma \ref{Kpp} guarantees that for every $i=1,\ldots,p$, $\underline{S}(x_i,\beta) = \underline{S}(y_i,\beta) = \left\lfloor \frac{i}{2} \right\rfloor + 1$.
    \item Finally, we need a proper edge-coloring $\gamma$ of $K_{2^{q+1}+1}$ with $3\cdot2^q$ colors. We denote its vertices by $V(K_{2^{q+1}+1}) = \left\{w^j : j=0,\ldots, 2^{q+1} \right\}$. The coloring $\gamma$ must satisfy the following condition: $def(K_{2^{q+1}+1},\gamma) = def(w^0, \gamma) = 2^q$. Colorings with these properties are described in the proof of Theorem 4.2 in \cite{GiaroKubaleMalaf2} and in the proof of Theorem 29 in \cite{PetMkhitaryan}.
\end{enumerate}

We denote the vertices of $K_{p2^{q+1}+1}$ in the following way (Fig. \ref{complete-graph}):
\begin{align*}
V\left(K_{p2^{q+1}+1}\right) = \{v_0\} \cup \left\{v_i^j : i=1,\ldots,p,\ j=1,\ldots,2^{q+1} \right\}.
\end{align*}

We construct its edge-coloring $\phi$ in the following way:
\begin{align*}
\phi\left(v_{i_1}^j v_{i_2}^j\right) &= \alpha(u_{i_1}u_{i_2}) + p\left(\gamma(w^0w^j) - 1\right), & 0 \leq i_1 < i_2 \leq p, & & 1 \leq j \leq 2^{q+1}; \\
\phi\left(v_{i_1}^{j_1}v_{i_2}^{j_2}\right) &= \beta(x_{i_1}y_{i_2}) + p\left(\gamma(w^{j_1}w^{j_2}) - 1\right), & 1 \leq i_1, i_2 \leq p, & & 1 \leq j_1 < j_2 \leq 2^{q+1}.
\end{align*}

The symbols $v_0^j$ in the above formulas refer to the same vertex $v_0$, for all $j=1,\ldots,2^{q+1}$. The first formula colors the edges that have both ends in the same group in Fig. \ref{complete-graph}, while the second formula colors the edges that join vertices from different groups. We show that the spectrums of the vertices $v_i^j$, $i=1,\ldots,p$, $j=1,\ldots,2^{q+1}$ are intervals.
\begin{align*}
    S\left(v^j_i,\phi\right) &= \left[\underline{S}(u_i, \alpha) + p\left(\gamma(w^0w^j) - 1\right), \overline{S}(u_i, \alpha) + p\left(\gamma(w^0w^j) - 1\right)\right] \\
    &\cup \bigcup\limits_{j'\in [1, 2^{q+1}] \setminus \{j\}}{\left[\underline{S}(x_i, \beta) + p\left(\gamma(w^jw^{j'}) - 1\right), \overline{S}(x_i, \beta) + p\left(\gamma(w^jw^{j'}) - 1\right) \right]}.
\end{align*}
We have that $\underline{S}(u_i, \alpha) = \underline{S}(x_i, \beta) = \lfloor \frac{i}{2} \rfloor + 1$ for every $i = 1,\ldots,p$. So the above expression becomes:
\begin{align*}
    S\left(v^j_i,\phi\right) &= \left[\left\lfloor \frac{i}{2} \right\rfloor + 1 + p\left(\gamma(w^0w^j) - 1\right), \left\lfloor \frac{i}{2} \right\rfloor + p + p\left(\gamma(w^0w^j) - 1\right)\right] \\
    &\cup \bigcup\limits_{j'\in [1, 2^{q+1}] \setminus \{j\}}{\left[\left\lfloor \frac{i}{2} \right\rfloor + 1 + p\left(\gamma(w^jw^{j'}) - 1\right), \left\lfloor \frac{i}{2} \right\rfloor + p + p\left(\gamma(w^jw^{j'}) - 1\right) \right]}\\
    &= \bigcup\limits_{j'\in [0, 2^{q+1}] \setminus \{j\}}{\left[\left\lfloor \frac{i}{2} \right\rfloor + 1 + p\left(\gamma(w^jw^{j'}) - 1\right), \left\lfloor \frac{i}{2} \right\rfloor + p + p\left(\gamma(w^jw^{j'}) - 1\right) \right]}\\
    &= \left[\left\lfloor \frac{i}{2} \right\rfloor + 1 + p\left(\underline{S}(w^j,\gamma) - 1\right), \left\lfloor \frac{i}{2} \right\rfloor + p + p\left(\overline{S}(w^j,\gamma) - 1\right) \right].
\end{align*}
The last equation holds because the spectrum $S\left(w^j, \gamma\right)$ is an interval for $j=1,\ldots,2^{q+1}$. Finally, the spectrum of $v_0$ and the deficiency of $\phi$ at $v_0$ will be:
\begin{align*}
    S\left(v_0,\phi\right) &= \bigcup\limits_{j'\in [1, 2^{q+1}]} {\left[ \underline{S}(u_0, \alpha) + p\left(\gamma(w^0w^{j'}) - 1\right), \overline{S}(u_0, \alpha) + p\left(\gamma(w^0w^{j'}) - 1\right) \right]}\\
    &= \bigcup\limits_{j'\in [1, 2^{q+1}]} {\left[ 1 + p\cdot\gamma(w^0w^{j'}) - p, p\cdot\gamma(w^0w^{j'}) \right]},\\
    def(v_0, \phi) &= \overline{S}(v_0, \phi) - \underline{S}(v_0, \phi)-d_{K_{2n+1}}(v_0) + 1\\
    &= p\cdot\overline{S}(w^0,\gamma) - (1 + p\cdot\underline{S}(w^0,\gamma) - p)-p\cdot2^{q+1} + 1\\
    &= p\left(\overline{S}(w^0,\gamma) - \underline{S}(w^0,\gamma)-d_{K_{2^{q+1}+1}}(w^0) + 1\right) = p\cdot def(w^0,\gamma) = p2^q = n.
\end{align*}

Let $w^{\underline{j}}$ and $w^{\overline{j}}$ be the vertices of $K_{2^{q+1}+1}$ such that $\underline{S}\left(w^{\underline{j}}, \gamma\right) = 1$ and $\overline{S}\left(w^{\overline{j}}, \gamma\right) = 3\cdot2^q$. It is easy to see that $\underline{S}\left(v^{\underline{j}}_{1},\phi\right) = 1$ and $\overline{S}\left(v^{\overline{j}}_{p},\phi\right) = \left\lfloor \frac{p}{2} \right\rfloor + p + p\left(3\cdot2^q - 1\right) = 3n + \frac{p-1}{2}$. So, $\phi$ is a proper edge-coloring of $K_{2n+1}$ with $3n + \frac{p-1}{2}$ colors having $def(K_{2n+1}, \phi) = def\left(v_0, \phi\right) = n$.
\end{proof}

The last theorem of this section generalizes Theorem \ref{mytheorem1.2}.

\begin{theorem}
\label{mytheorem2.5}
Let $\alpha_0$ be a proper $t_0$-edge-coloring of a regular graph $G$ and $D \subseteq V(G)$ be a subset of its vertices. If for every $v \in V(G) \setminus D$, $def(v,\alpha_0)=0$, then for every $t$, $\overline{S}(D,\alpha_0) - \underline{S}(D,\alpha_0) + 1 \leq t \leq t_0$, $G$ has a proper $t$-edge-coloring $\alpha$ such that $def(v,\alpha) = def(v,\alpha_0)$ for every vertex $v\in V(G)$.
\end{theorem}
\begin{proof}
Let $a = \underline{S}(D,\alpha_0) - 1$ and $b = t_{0} - \overline{S}(D,\alpha_0)$. Basically we have $a+b$ colors outside the range $[\underline{S}(D,\alpha_0),\overline{S}(D,\alpha_0)]$ and we need to get rid of $t_0 - t$ of them. We construct the edge-coloring $\alpha$ of $G$ by copying colors of all the edges from $\alpha_0$ and then modifying some of them. First we try to remove the colors larger than $\overline{S}(D,\alpha_0)$, and if these are not enough, we remove the colors smaller than $\underline{S}(D,\alpha_0)$.

If $t_0 - t \leq b$, then for every $e\in E(G)$ with $\alpha_0(e)\in \left[t + 1, t_0\right]$, we set $\alpha(e) = \alpha_0(e) - \Delta(G)$.

If $t_0 - t > b$, then for every $e\in E(G)$ with $\alpha_0(e)\in \left[\overline{S}(D, \alpha_0) + 1, t_0\right]$, we set $\alpha(e) = \alpha_0(e) - \Delta(G)$. Next, for every $e\in E(G)$ with $\alpha_0(e)\in \left[1, t_0 - t - b \right]$, we set $\alpha(e) = \alpha_0(e) + \Delta(G)$.
In both cases, the spectrums of the vertices from $D$ are not modified and the spectrums of the other vertices remain intervals. If $\underline{S}(V(G), \alpha) > 1$, then we subtract $\underline{S}(V(G), \alpha) - 1$ from colors of all the edges to obtain the final coloring.
\end{proof}

Note that if we apply the above theorem on an interval colorable regular graph $G$ by taking its interval $W(G)$-coloring as $\alpha_0$ and any set containing a single vertex as $D$, we obtain Theorem \ref{mytheorem1.2}.

Theorems \ref{mytheorem2.4} and \ref{mytheorem2.5} imply the following result.

\begin{corollary}
\label{mycorollary2.8}
Let $n=p2^q$, where $p$ is odd and $q \in \mathbb{Z}_+$. For every $t$, $3n \leq t \leq 3n + \frac{p-1}{2}$, $K_{2n+1}$ has a proper $t$-edge-coloring $\alpha$ such that $def(K_{2n+1},\alpha)=n$.
\end{corollary}
\begin{proof} It is enough to take a proper $\left(3n + \frac{p-1}{2}\right)$-edge-coloring $\alpha_0$ of the graph $K_{2n+1}$ from the proof of Theorem \ref{mytheorem2.4} and a set $D=\{v_0\}$, where $v_0\in V(K_{2n+1})$ and $def(v_0,\alpha_0)=def(K_{2n+1},\alpha_0)=n$, and notice that $\overline{S}(D,\alpha_0) - \underline{S}(D,\alpha_0) + 1=\overline{S}(v_0,\alpha_0) - \underline{S}(v_0,\alpha_0) + 1= def(v_0,\alpha_0)+d_{K_{2n+1}}(v_0) =n+2n =3n$.
\end{proof}
\bigskip

\section{The deficiency of certain graphs}\

This section is devoted to the deficiency of near-complete graphs
and some complete tripartite graphs. We begin our consideration with near-complete graphs. Recently, it was shown that the following result holds.

\begin{theorem}
\label{mytheorem3.1}\cite{B-OD-BHal} If $G$ is a graph with an odd number of vertices, then
\begin{center}
$def(G)\geq \frac{2\vert E(G)\vert -(\vert V(G)\vert
-1)\Delta(G)}{2}$.
\end{center}
\end{theorem}

\begin{corollary}
\label{mycorollary3.1}\cite{B-OD-BHal} For any $n,m\in \mathbb{N}$, we have
\begin{center}
$def(K_{2n+1}-mK_{2})\geq n-m$, where $1\leq m\leq n$.
\end{center}
\end{corollary}

In general, the lower bound on $def(K_{2n+1}-mK_{2})$ ($1\leq m\leq
n$) is not sharp, since for any maximum matching $M$ of $K_{2n+1}$
($n\geq 2$), $K_{2n+1} - M\notin \mathfrak{N}$ \cite{PetDM}, hence
$def(K_{2n+1}-nK_{2})>0$. However, the lower bound is sharp for
$K_{2n+1} - e$ as it was conjectured in \cite{B-OD-BHal}.

\begin{theorem}
\label{mytheorem3.2} If $n\in \mathbb{N}$, then $def(K_{2n+1}-e)=n-1$.
\end{theorem}
\begin{proof}
By Corollary \ref{mycorollary3.1}, we have $def(K_{2n+1}-e)\geq n-1$ for any $n\in \mathbb{N}$. For the proof, it suffices to construct a proper edge-coloring $\beta$ of $K_{2n+1}-e$ with deficiency $def(K_{2n+1}-e,\beta)=n-1$. Let
$V(K_{2n+1}-e)=\left\{v_{0},v_{1},\ldots,v_{2n}\right\}$. Without
loss of generality we may assume that $e=v_{1}v_{2n}$.

Define an edge-coloring $\beta$ of $K_{2n+1}-e$. For each edge
$v_{i}v_{j}\in E(K_{2n+1})$ with $i<j$ and $(i,j)\neq (1,2n)$,
define a color $\beta\left(v_{i}v_{j}\right)$ as follows:\\

$\beta\left(v_{i}v_{j}\right)=\left\{
\begin{tabular}{ll}
$1$, & if $i=0$, $j=1$;\\
$2n+1$, & if $i=0$, $j=2$;\\
$j-1$, & if $i=0$, $3\leq j\leq n$;\\
$n+1+j$, & if $i=0$, $n+1\leq j\leq 2n-2$;\\
$n$, & if $i=0$, $j=2n-1$;\\
$2n$, & if $i=0$, $j=2n$;\\
$i+j-1$, & if $1\leq i\leq \left\lfloor\frac{n}{2}\right\rfloor$,
$2\leq j\leq n$, $i+j\leq n+1$;\\
$i+j+n-2$, & if $2\leq i\leq n-1$, $\left\lfloor
\frac{n}{2}\right\rfloor +2\leq j\leq n$, $i+j\geq n+2$;\\
$n+1+j-i$, & if $3\leq i\leq n$, $n+1\leq j\leq 2n-2$, $j-i\leq
n-2$;\\
$j-i+1$, & if $1\leq i\leq n$, $n+1\leq j\leq 2n$, $j-i\geq n$;\\
$2i-1$, & if $2\leq i\leq 1+\left\lfloor
\frac{n-1}{2}\right\rfloor$, $n+1\leq j\leq n+\left\lfloor
\frac{n-1}{2}\right\rfloor$, $j-i=n-1$;\\
$i+j-1$, & if $\left\lfloor \frac{n-1}{2}\right\rfloor +2\leq i\leq
n$, $n+1+\left\lfloor \frac{n-1}{2}\right\rfloor\leq j\leq 2n-1$,
$j-i=n-1$;\\
$i+j-2n+1$, & if $n+1\leq i\leq n+\left\lfloor
\frac{n}{2}\right\rfloor -1$, $n+2\leq j\leq 2n-2$, $i+j\leq
3n-1$;\\
$i+j-n$, & if $n+1\leq i\leq 2n-1$, $n+\left\lfloor
\frac{n}{2}\right\rfloor +1\leq j\leq 2n$, $i+j\geq 3n$.
\end{tabular}%
\right.$\\

Let us prove that $\beta $ is a proper $(3n-1)$-edge-coloring of
$K_{2n+1}-e$ with deficiency $def(K_{2n+1}-e,\beta)=n-1$.

Let $G$ be the subgraph of $K_{2n+1}-e$ induced by
$\{v_{1},\ldots,v_{2n}\}$. Clearly, $G$ is isomorphic to $K_{2n}-e$.
This edge-coloring $\beta$ of $K_{2n+1}-e$ is constructed on the interval $(3n-2)$-coloring of $K_{2n}$ which is described in
the proof of Theorem 4 from \cite{PetDM}. We use this interval
$(3n-2)$-coloring of $K_{2n}$ and then we shift all colors of the
edges of $G$ by one. Let $\alpha$ be this edge-coloring of $G$.
Using the property of this edge-coloring which is described in the
proof of Corollary 6 from \cite{PetDM}, we get

\begin{description}
\item[1)] $S\left(v_{1},\alpha\right)=[2,2n-1]$ and $S\left(v_{2},\alpha\right)=[2,2n]$,

\item[2)] $S\left(v_{i},\alpha\right)=S\left(v_{n+i-2},\alpha\right)=[i,2n-2+i]$ for $3\leq i\leq n$,

\item[3)]
$S\left(v_{2n-1},\alpha\right)=[n+1,3n-1]$ and
$S\left(v_{2n},\alpha\right)=[n+1,3n-1]\setminus\{2n\}$.
\end{description}

Now, by the definition of $\beta$, we have

\begin{description}
\item[1)] $S\left(v_{0},\beta\right)=[1,n]\cup [2n,3n-1]$,

\item[2)] $S\left(v_{1},\beta\right)=[1,2n-1]$ and $S\left(v_{2},\beta\right)=[2,2n+1]$,

\item[3)] $S\left(v_{i},\beta\right)=[i-1,2n-2+i]$ and $S\left(v_{n+i-2},\beta\right)=[i,2n-1+i]$ for $3\leq i\leq n$,

\item[4)]
$S\left(v_{2n-1},\beta\right)=[n,3n-1]$ and
$S\left(v_{2n},\beta\right)=[n+1,3n-1]$.
\end{description}

This shows that $\beta $ is a proper $(3n-1)$-edge-coloring of
$K_{2n+1}-e$ with deficiency $def(K_{2n+1}-e,\beta)=n-1$. Hence,
$def(K_{2n+1}-e)=n-1$. (see Fig. 2).
\end{proof}

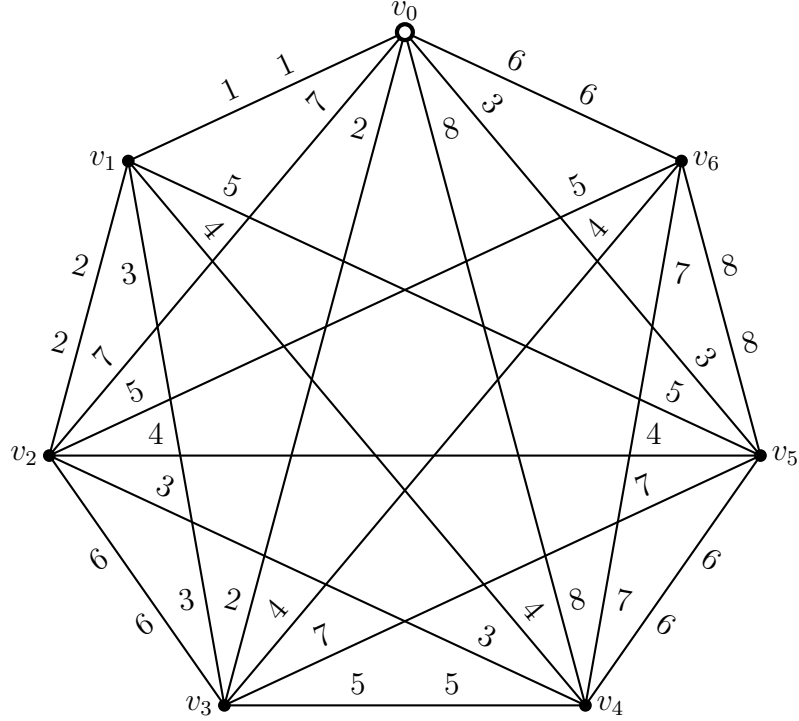
\begin{figure}
  \begin{center}
    % \tikzset{near start abs/.style={pos=1cm}}
    % \tikzset{near end abs/.style={pos=-1cm}}
  \begin{tikzpicture}[style=thick]
    \coordinate (V3) at (-120:4.8cm);
    \coordinate (V2) at (-170:4.8cm);
    \coordinate (V1) at (140:4.8cm);
    \coordinate (V0) at (90:4.8cm);
    \coordinate (V4) at (-60:4.8cm);
    \coordinate (V5) at (-10:4.8cm);
    \coordinate (V6) at (40:4.8cm);

    \draw (V0) -- node [sloped,above,pos=0.4] {$1$} node [sloped,above,pos=0.6] {$1$} (V1);
    \draw (V0) -- node [sloped,above,pos=0.2] {$7$} node [sloped,above,pos=0.8] {$7$} (V2);
    \draw (V0) -- node [sloped,left,rotate=-90,pos=0.15] {$2$} node [sloped,left,rotate=-90,pos=0.85] {$2$} (V3);
    \draw (V0) -- node [sloped,right,rotate=90,pos=0.15] {$8$} node [sloped,right,rotate=90,pos=0.85] {$8$} (V4);
    \draw (V0) -- node [sloped,above,pos=0.2] {$3$} node [sloped,above,pos=0.8] {$3$} (V5);
    \draw (V0) -- node [sloped,above,pos=0.37] {$6$} node [sloped,above,pos=0.63] {$6$} (V6);

    \draw (V1) -- node [sloped,left,rotate=-90,pos=0.37] {$2$} node [sloped,left,rotate=-90,pos=0.63] {$2$} (V2);
    \draw (V1) -- node [sloped,left,rotate=90,pos=0.2] {$3$} node [sloped,left,rotate=90,pos=0.8] {$3$} (V3);
    \draw (V1) -- node [sloped,above,pos=0.15] {$4$} node [sloped,above,pos=0.85] {$4$} (V4);
    \draw (V1) -- node [sloped,above,pos=0.15] {$5$} node [sloped,above,pos=0.85] {$5$} (V5);
    % \draw (V1) -- node [sloped,left,rotate=-90] {$0$} (V6);

    \draw (V2) -- node [sloped,left,rotate=90,pos=0.37] {$6$} node [sloped,left,rotate=90,pos=0.63] {$6$} (V3);
    \draw (V2) -- node [sloped,above,pos=0.2] {$3$} node [sloped,above,pos=0.8] {$3$} (V4);
    \draw (V2) -- node [sloped,above,pos=0.15] {$4$} node [sloped,above,pos=0.85] {$4$} (V5);
    \draw (V2) -- node [sloped,above,pos=0.15] {$5$} node [sloped,above,pos=0.85] {$5$} (V6);

    \draw (V3) -- node [sloped,above,pos=0.37] {$5$} node [sloped,above,pos=0.63] {$5$} (V4);
    \draw (V3) -- node [sloped,above,pos=0.2] {$7$} node [sloped,above,pos=0.8] {$7$} (V5);
    \draw (V3) -- node [sloped,above,pos=0.15] {$4$} node [sloped,above,pos=0.85] {$4$} (V6);

    \draw (V4) -- node [sloped,right,rotate=-90,pos=0.37] {$6$} node [sloped,right,rotate=-90,pos=0.63] {$6$} (V5);
    \draw (V4) -- node [sloped,right,rotate=-90,pos=0.2] {$7$} node [sloped,right,rotate=-90,pos=0.8] {$7$} (V6);

    \draw (V5) -- node [sloped,right,rotate=90,pos=0.37] {$8$} node [sloped,right,rotate=90,pos=0.63] {$8$} (V6);

    \draw[fill=white,style=ultra thick] (V0) circle (3pt) node [above] {$v_0$};
    \draw[fill=black] (V1) circle (2pt) node [left] {$v_1$};
    \draw[fill=black] (V2) circle (2pt) node [left] {$v_2$};
    \draw[fill=black] (V3) circle (2pt) node [left] {$v_3$};
    \draw[fill=black] (V4) circle (2pt) node [right] {$v_4$};
    \draw[fill=black] (V5) circle (2pt) node [right] {$v_5$};
    \draw[fill=black] (V6) circle (2pt) node [right] {$v_6$};
  \end{tikzpicture}
  \end{center}
  \caption{Proper $8$-edge-coloring $\beta$ of $K_7-e$ with $def( v_0,\beta)=def(K_7-e)=2$.}
  \label{K7-minus-edge}
\end{figure}

\begin{theorem}
\label{mytheorem3.3} If $n\in \mathbb{N}$, then
$w_{def}(K_{2n+1}-e)=3n-1$.
\end{theorem}
\begin{proof}
From the proof of Theorem \ref{mytheorem3.2}, we have that
$K_{2n+1}-e$ has a proper $(3n-1)$-edge-coloring $\alpha$ with
deficiency $def(K_{2n+1}-e,\alpha)=def(K_{2n+1}-e)=n-1$. This
implies that $w_{def}(K_{2n+1}-e)\leq 3n-1$. On the other hand,
since $K_{2n+1}-e$ has no perfect matching,
$\delta(K_{2n+1}-e)=2n-1$ and $def(K_{2n+1}-e)=n-1$, by Theorem
\ref{mytheorem2.3}, we obtain $w_{def}(K_{2n+1}-e)\geq
2(2n-1)-n+1=3n-1$.
\end{proof}

In \cite{FengHuang}, it was shown that $def\left(K_{1,1,n}\right)=0$ if $n$ is even, and $def\left(K_{1,1,n}\right)=1$ if $n$ is odd.
Here we generalize this result and determine the deficiency of
$K_{1,m,n}$ for any $m,n\in \mathbb{N}$.

\begin{theorem}
\label{mytheorem3.4} For any $m,n\in \mathbb{N}$, we have
\begin{center}
$def\left(K_{1,m,n}\right)=\left\{
\begin{tabular}{ll}
$0$, & if $\gcd(m+1,n+1)=1$,\\
$1$, & otherwise.\\
\end{tabular}%
\right.$
\end{center}
\end{theorem}
\begin{proof} By Theorem \ref{mytheorem1.5}, we have $K_{1,m,n}\in \mathfrak{N}$ if and only if
$\gcd(m+1,n+1)=1$ for any $m,n\in \mathbb{N}$. This implies that
$def\left(K_{1,m,n}\right)=0$ if $\gcd(m+1,n+1)=1$, and
$def\left(K_{1,m,n}\right)\geq 1$ if $\gcd(m+1,n+1)>1$. We now show that $def\left(K_{1,m,n}\right)\leq 1$.

Let
$V\left(K_{1,m,n}\right)=\{u_{1},\ldots,u_{m},v_{1},\ldots,v_{n},w\}$
and $E\left(K_{1,m,n}\right)=\left\{u_{i}v_{j}\colon\,1\leq i\leq
m,1\leq j\leq n\right\}\cup \{wu_{i}\colon\,1\leq i\leq m\}\cup
\{wv_{j}\colon\,1\leq j\leq n\}$.\\

Define an edge-coloring $\alpha$ of $K_{1,m,n}$ as follows:

\begin{description}
\item[(1)] for $1\leq i\leq m$ and $1\leq j\leq n$, let
$\alpha\left(u_{i}v_{j}\right)=i+j$,

\item[(2)] for $1\leq i\leq m$, let $\alpha\left(wu_{i}\right)=i$,

\item[(3)] for $1\leq j\leq n$, let $\alpha\left(wv_{j}\right)=m+1+j$.
\end{description}

Let us prove that $\alpha$ is a proper $(m+n+1)$-edge-coloring of
$K_{1,m,n}$ with deficiency $def\left(K_{1,m,n},\alpha\right)=1$.\\

By the definition of $\alpha$, we have
\begin{description}
\item[1)] for $1\leq i\leq m$,

$S\left(u_{i},\alpha\right)=[i+1,n+i]\cup \{i\}=[i,n+i]$ due to (1) and (2),

\item[2)] for $1\leq j\leq n$,

$S\left(v_{j},\alpha\right)=[j+1,m+j]\cup \{m+1+j\}=[j+1,m+1+j]$ due to (1) and (3),

\item[3)] $S\left(w,\alpha\right)=[1,m]\cup [m+2,m+n+1]$
due to (2) and (3).
\end{description}

This implies that $\alpha$ is a proper $(m+n+1)$-edge-coloring of
$K_{1,m,n}$ with deficiency $def\left(K_{1,m,n},\alpha\right)=1$
($m,n\in \mathbb{N}$); thus $def\left(K_{1,m,n}\right)\leq 1$ if
$\gcd(m+1,n+1)>1$.
\end{proof}
\bigskip

\section{Concluding Remarks}\

In 1999, Giaro, Kubale and Ma\l afiejski \cite{GiaroKubaleMalaf1} showed that there are bipartite graphs whose deficiency approaches the number of vertices. On the other hand, in \cite{BouchHertzDesau} Bouchard, Hertz, Desaulniers suggested the following conjecture.

\begin{conjecture} For every graph $G$,
\begin{center}
$def(G)\leq \vert V(G)\vert$.
\end{center}
\end{conjecture}

The conjecture is still open, but it holds for regular graphs, some bipartite graphs and outerplanar graphs \cite{Schwartz,GiaroKubaleMalaf1,B-OD-B,KhachOuterplanar}. If the conjecture is true, then, by Corollary \ref{mycorollary2.1}, we obtain that the upper bound $W_{def}(G)\leq 2\vert V(G)\vert -1$ holds for every  triangle-free graph $G$. Moreover, if the conjecture is true, then, by the result of Altinakar, Caporossi and Hertz \cite{AltinCaporHertz}, we derive that the bound $W_{def}(G)\leq 3\vert V(G)\vert -4$ holds for every graph $G$ with at least three vertices.

In Section 3, we obtained some results on the parameters $w_{def}(K_{2n+1})$ and $W_{def}(K_{2n+1})$. In particular, we proved that the difference $W_{def}(K_{2n+1}) - w_{def}(K_{2n+1})$ can be arbitrarily large. On the other hand, we cannot find a proper edge-coloring of $K_{2n+1}$ with more than $3n$ colors having minimum deficiency when $n$ is a power of two. So we would like to suggest the following conjecture.

\begin{conjecture} For any $q\in \mathbb{Z}_+$,
\begin{center}
$w_{def}\left(K_{2^{q+1}+1}\right)=W_{def}\left(K_{2^{q+1}+1}\right)=3\cdot2^q$.
\end{center}
\end{conjecture}

Finally, we would like to introduce the concept of deficiency-critical graphs, since it will be very useful for investigating deficiency of graphs. A graph $G$ with $def(G)=k>0$  is \emph{$k$-deficiency-critical} if $def(G-e)<def(G)$ for every edge $e$ in $G$. Clearly, odd cycles $C_{2n+1}$ ($n\in \mathbb{N}$) are $1$-deficiency-critical graphs. In fact, the conjecture of Borowiecka-Olszewska, Drgas-Burchardt and Ha\l uszczak can be reformulated as follows: all complete graphs $K_{2n+1}$ ($n\in \mathbb{N}$) are $n$-deficiency-critical. Here we would like to suggest the following problem.

\begin{problem}
Characterize all $k$-deficiency-critical graphs.
\end{problem}
\bigskip

\begin{acknowledgement}
We would like to thank both referees for many useful comments and suggestions. The work was made possible by a research grant from the Armenian National Science and Education Fund (ANSEF) based in New York, USA.
\end{acknowledgement}


\bigskip

\begin{thebibliography}{99}


\bibitem{AltinCaporHertz} S. Altinakar, G. Caporossi, A. Hertz, A
comparison of integer and constraint programming models for the
deficiency problem, Computers and Oper. Res. 68 (2016) 89-96.

\bibitem{AsrKam} A.S. Asratian, R.R. Kamalian, Interval colorings of edges of a
multigraph, Appl. Math. 5 (1987) 25-34 (in Russian).

\bibitem{AsrKamJCTB} A.S. Asratian, R.R. Kamalian, Investigation on interval
edge-colorings of graphs, J. Combin. Theory Ser. B 62 (1994) 34-43.

\bibitem{AsrDenHag} A.S. Asratian, T.M.J. Denley, R. Haggkvist, Bipartite Graphs and their Applications, Cambridge University Press, Cambridge, 1998.

\bibitem{Axen} M.A. Axenovich, On interval colorings of planar graphs, Congr. Numer. 159 (2002)
77-94.

\bibitem{BeinekeWilson} L.W. Beineke, R.J. Wilson, On the
edge-chromatic number of a graph, Discrete Math. 5 (1973) 15-20.

\bibitem{B-OD-BHal} M. Borowiecka-Olszewska, E. Drgas-Burchardt, M. Ha\l uszczak, On the structure and deficiency of $k$-trees with bounded degree, Discrete Appl. Math. 201 (2016) 24-37.

\bibitem{B-OD-B} M. Borowiecka-Olszewska, E. Drgas-Burchardt, The deficiency of all generalized Hertz graphs and minimal consecutively non-colourable graphs in this class, Discrete Math. 339 (2016) 1892-1908.

\bibitem{BouchHertzDesau} M. Bouchard, A. Hertz, G. Desaulniers, Lower bounds and a tabu search algorithm for the minimum deficiency problem, J. Comb. Optim. 17 (2009) 168-191.

\bibitem{FengHuang} Y. Feng, Q. Huang, Consecutive edge-coloring of the generalized
$\theta$-graph, Discrete Appl. Math. 155 (2007) 2321-2327.

\bibitem{GiaroKubale1} K. Giaro, M. Kubale, Consecutive edge-colorings of complete and incomplete Cartesian products of graphs, Cong. Num. 128 (1997) 143-149.

\bibitem{GiaroKubale2} K. Giaro, M. Kubale, Compact scheduling of zero-one time operations
in multi-stage systems, Discrete Appl. Math. 145 (2004) 95-103.

\bibitem{GiaroKubaleMalaf1} K. Giaro, M. Kubale, M. Ma\l afiejski, On the deficiency of bipartite graphs, Discrete Appl. Math. 94 (1999) 193-203.

\bibitem{GiaroKubaleMalaf2} K. Giaro, M. Kubale, M. Ma\l afiejski, Consecutive colorings of the edges of general graphs, Discrete Math. 236 (2001) 131-143.

\bibitem{GrzesikKhach} A. Grzesik, H. Khachatrian, Interval edge-colorings of $K_{1,m,n}$, Discrete Appl. Math. 174 (2014) 140-145.

\bibitem{Hansen} H.M. Hansen, Scheduling with minimum waiting periods, MSc Thesis, Odense University, Odense, Denmark, 1992 (in Danish).

\bibitem{HansonLotenToft} D. Hanson, C.O.M. Loten, B. Toft, On interval colorings of bi-regular bipartite graphs, Ars Combin. 50 (1998) 23-32.

\bibitem{Kampreprint} R.R. Kamalian, Interval colorings of complete bipartite graphsand trees, preprint, Comp. Cen. of Acad. Sci. of Armenian SSR, Yerevan, 1989 (in Russian).

\bibitem{KamDiss} R.R. Kamalian, Interval edge colorings of graphs, Doctoral Thesis, Novosibirsk, 1990.

\bibitem{KamMir} R.R. Kamalian, A.N. Mirumian, Interval edge colorings of bipartite graphs of some class, Dokl. NAN RA 97 (1997) 3-5 (in Russian).

\bibitem{KamPet1} R.R. Kamalian, P.A. Petrosyan, A note on interval edge-colorings of graphs, Math. Probl. Comput. Sci. 36 (2012) 13-16.

\bibitem{KamPet2} R.R. Kamalian, P.A. Petrosyan, A note on upper bounds for the maximum span in interval edge-colorings of graphs, Discrete Math. 312 (2012) 1393-1399.

\bibitem{KhachOuterplanar} H.H. Khachatrian, Deficiency of outerplanar graphs, Proceedings of the Yerevan State University (2017), in press.

\bibitem{Kubale} M. Kubale, Graph Colorings, American Mathematical Society, 2004.

\bibitem{PetDM} P.A. Petrosyan, Interval edge-colorings of complete graphs and $n$-dimensional cubes, Discrete Math. 310 (2010) 1580-1587.

\bibitem{PetDMGT} P.A. Petrosyan, Interval edge colorings of some products of graphs, Discuss. Math. Graph Theory 31(2) (2011) 357-373.

\bibitem{PetKarapet} P.A. Petrosyan, G.H. Karapetyan, Lower bounds for the greatest possible number of colors in interval edge colorings of bipartite cylinders and bipartite tori, in: Proceedings of the CSIT Conference
(2007) 86-88.

\bibitem{PetKhachTan} P.A. Petrosyan, H.H. Khachatrian, H.G. Tananyan, Interval edge-colorings of Cartesian products of graphs I, Discuss. Math. Graph Theory 33(3) (2013) 613-632.

\bibitem{PetKhachYepTan} P.A. Petrosyan, H.H. Khachatrian, L.E. Yepremyan, H.G. Tananyan, Interval edge-colorings of graph products, in: Proceedings of the CSIT Conference (2011) 89-92.

\bibitem{PetKhachMam} P.A. Petrosyan, H.H. Khachatrian, T.K.
Mamikonyan, On interval edge-colorings of bipartite graphs, IEEE
Computer Science and Information Technologies (CSIT) 2015, 71-76.

\bibitem{PetMkhitaryan} P.A. Petrosyan, S.T. Mkhitaryan, Interval cyclic edge-colorings of graphs, Discrete Math. 339 (2016) 1848-1860.

\bibitem{PetSarg} P.A. Petrosyan, H.E. Sargsyan, On resistance of graphs, Discrete Appl. Math. 159 (2011) 1889-1900.

\bibitem{Seva} S.V. Sevast'janov, Interval colorability of the edges of a bipartite graph, Metody Diskret. Analiza 50 (1990) 61-72 (in Russian).

\bibitem{Schwartz} A. Schwartz, The deficiency of a regular graph, Discrete Math. 306 (2006) 1947-1954.

\bibitem{TePet} H.H. Tepanyan, P.A. Petrosyan, Interval edge-colorings of composition of graphs, Discrete Appl. Math. 217 (2017) 368-374.

\bibitem{Vizing} V.G. Vizing, The chromatic class of a multigraph, Kibernetika 3 (1965) 29-39 (in Russian).

\bibitem{West} D.B. West, Introduction to Graph Theory, Prentice-Hall, New Jersey, 2001.

\end{thebibliography}
\end{document}